\documentclass{amsart}
\usepackage{amssymb}
\usepackage{amsthm}
\usepackage[dvipsnames]{xcolor}
\usepackage{verbatim}
\usepackage{graphicx}
\newtheorem{thm}{Theorem}[section]
\newtheorem{cor}[thm]{Corollary}

\newtheorem{exam}[thm]{Example}
\theoremstyle{definition}
\theoremstyle{remark}
\newtheorem{rem}[thm]{Remark}
\numberwithin{equation}{section}

\begin{document}

\title[$J$-class  weighted translations on L. C. G. ]
{$J$-class weighted translations on locally compact groups}

\author{\bf  A. R. Imanzadeh Fard$^1$, M. R. Azimi$^{*^2}$ and    I.
Akbarbaglu$^3$  }

\address{A. R. Imanzadeh Fard}
\email{a.imanzadeh98@gmail.com}

\address {M. R. Azimi}
\email{mhr.azimi@maragheh.ac.ir(for corresponding)}

\address{I. Akbarbaglu}
\email{i.akbarbaglu@cfu.ac.ir}

\address{$1, 2$: Department of Mathematics, Faculty of Sciences,
University of Maragheh, P.O. Box: 55181-83111, Golshahr, Maragheh,
Iran}

\address{ $^{3}$ Department of Mathematics Education,
Farhangian University, P.O. Box: 14665-889, Tehran, Iran}

\thanks{$^*$Corresponding author}

\subjclass[2020]{Primary 47A16; Secondary 47B38.}

\keywords{Extended limit set, $J$-class, Orbit, Hypercyclic,
Weighted translation, Locally compact group, Haar measure.}

\date{}

\dedicatory{}

\commby{}


\begin{abstract}
A bounded linear operator $T$ on a Banach space $X$ (not necessarily
separable) is said to be $J$-class operator whenever the extended
limit set, say $J_T(x)$ equals $X$ for some vector $x\in X$.
Practically, the extended limit sets localize the dynamical behavior
of operators.
 In this paper, using the extended limit sets we will examine the necessary
 and sufficient conditions for the  weighted translation $T_{a,\omega}$
  to be $J$-class  on a locally compact group $G$,
 within the setting of $ L^p$-spaces for $ 1 \leq p < \infty $.
 Precisely, we delineate the boundary between $J$-class and
hypercyclic behavior for weighted translations. Then, we will show
that for the torsion elements in  locally compact
 groups,  unlike the case of  non-dense orbits of  weighted
translations, we have $J_{T_{a,\omega}}(0)=L^p(G)$.
 Finally, we will provide some examples on which the  weighted translation
$ T_{a,\omega}$ is $J$-class but it fails to be hypercyclic.
\end{abstract}
\maketitle
\section{\textbf{Introduction}}
Let's assume that $ G $ is a locally compact group with an identity
element $ e $ and equipped with the right Haar measure $\lambda$.
 An element $ g \in G $ is called  \emph{torsion}  if it is of finite order i.e.,
  there exists
 $ n \in \mathbb{N} $ such that $ g^n = e $ and called
 \emph{periodic} if the closed subgroup generated by $g$ i.e.,
  $<g>=\{g^n: ~ n\in \mathbb{Z}\}$ is compact. It is clear
 that every  torsion element is periodic. Moreover, we
 call that $ g $ is \emph{aperiodic} if it is not periodic.
We say that an element $g\in G$ \emph{strongly runaway}
 whenever for each compact subset $K\subset G$,  there exists an $N \in
\mathbb{N}$ such that $K \cap Kg^{\pm m}= \emptyset$, for each  $ m
\geq N $. The interesting result with a pivotal role stated in
\cite[Lemma 2.1]{chu}, asserts that the elements of the second
countable groups are  aperiodic if and only if those are strongly
runaway.

 For an arbitrary element
 $g \in G $, the unit point mass  measure (or Dirac measure) at $g$, is denoted by $\delta_g$.
    For each  subset $A \subset G $, it is defined by  $ \delta_g(A) = 1 $ if
 $ g \in A $ and $ \delta_g(A) = 0$, otherwise.
 The Banach space $ L^p(G) $ (for $ 1 \leq p < \infty $) consists
of all Borel measurable functions $ f: G \rightarrow \mathbb{C}$
such that:
\[
\|f\|_p = \left( \int_G |f(x)|^p \, d\lambda(x)
\right)^{\frac{1}{p}} < \infty.
\]
A Borel measurable  function $ \omega: G \to (0, \infty) $ is called
a \emph{weight} on  $ G $. Given an element $ a \in G $, we define
typically the \emph{weighted translation operator} $ T_{a,\omega}:
L^p(G) \to L^p(G) $ for $ 1 \leq p < \infty $ by the assignment:
 $$ f\mapsto \omega f\ast\delta_a,\qquad\forall f\in L^p(G),$$
where, for every $ x \in G$, the convolution:
\begin{eqnarray*}
(f\ast\delta_a)(x)=\int_G\,\,f(xy^{-1})\,d\delta_a(y)=f(xa^{-1})
\end{eqnarray*}
is interpreted as  the right translation of $f$ by $a^{-1}$. That
the weighted translation $ T_{a,\omega}$ is well-defined it is
sufficient to assume that its weight should be essentially bounded
i.e., $\omega\in L^{\infty}(G)$.

For each Borel-measurable subset $A$, the \emph{change of variable
formula},  $$\int_{A} f*\delta_a(x)d\lambda(x)=
  \int_{Aa^{-1}} f(x) d\lambda(x),$$ will be used
  frequently throughout this paper. For
each $m\in \mathbb{N}$, the iterates of the weighted translation
 $T_{a,\omega}$ is obtained as follows:
\begin{eqnarray*}
T^m_{a,\omega}(f)(x)=[\prod_{i=0}^{m-1}\omega(xa^{-i})]f(xa^{-m})
\end{eqnarray*}
For simplicity of the notation, we adopt that:
 $${\tilde{\omega}_{m}}(x)=(\prod_{i=0}^{m-1}\omega(xa^{-i}))^{-1},
 \qquad \omega_{m}(x)=\prod_{i=1}^{m}\omega(xa^{i}).$$
 So
\begin{eqnarray*}
T^m_{a,\omega}(f)(x)={\tilde{\omega}_{m}}^{-1}(x)f(xa^{-m}).
\end{eqnarray*}

Let $X$ be a separable Banach space. A bounded linear operator $T:
X\rightarrow X$ is said to be \emph{hypercyclic} whenever there
exists a vector $x\in X$ such that its orbit, $$orb(T,x)=\{T^nx:
n=0,1,2,\cdot\cdot\cdot\}$$ is dense in $X$. Concerning dynamics of
the linear operators, the interested readers are referred to the
books \cite{baya, gros}.  In the year 1995, the work \cite{sal} done
by H. Salas was the first study of hypercyclicity of  weighted
shifts and the results were taken in terms of the weights. The
hypercyclicity and other dynamical behaviors of the weighted
translation $T_{a, \omega}$ on locally compact groups for various
cases of the element $a$, have been extensively investigated in
\cite{chu, chen, chen3}. They provided a necessary and sufficient
conditions for the hypercyclicity of $T_{a, \omega}$ when $a$ is an
aperiodic element. Afterwards,   G. Costakis and  A. Manoussos have
localized the concept of hypercyclicity in \cite{cos}. For this,
they introduced the limit sets and the extended limit sets. Given a
bounded linear operator $T$ on a Banach space $X$ (not necessarily
separable) and $x\in X$, the \emph{limit set} is defined as follows:
$$L_T(x)=\{y\in X | ~ \exists (n_k) \subseteq \mathbb{N},~
\text{s.t.}~ T^{n_k} x \to y ~ \text{as} ~ k\rightarrow \infty\}.$$
 Note that $ T $ is hypercyclic if there exists a vector
$ x \in X $ such that $L_T(x)=X$.  Additionally, we denote the
\emph{extended limit set} by $ J_T(x) $. It
 consists of all vectors $y$ in $X$ on which there exist a strictly increasing
  sequence of positive integers $(n_k)$ and a sequence $(x_k)\subseteq X$ such that
   $x_k \to x$  and $T^{n_k} x_k \to y$ as $k\rightarrow \infty$.  Shortly,
  \[
   J_T(x) = \{ y \in X :\exists \upuparrows(n_k) \subseteq \mathbb{N},   \exists (x_k)
   \subset X ~ \text{s.t. }~ ~ x_k \to x \text{ and } T^{n_k} x_k \to y \}
   .\]
An operator $ T $ is said to be  \textit{J-class } whenever $ J_T(x)
= X $ for some vector $x\in X$. In this case, a vector $ x $ is
called a \textit{J-vector} for $ T $. For all $x\in X$, the sets $
L_T(x)$  $ J_T(x)$ are  closed and $T$-invariant. In the case, an
operator $T$ is power bounded i.e., for each $n\in \mathbb{N}$,
$\|T^n\|\leq M$ (for some real number) then $ J_T(x)=L_T(x)$
(\cite{man}).

 It is worth noting that if an
operator $T$ is hypercyclic then $ J_T(x) = X $ for each vector
$x\in X$. Further,  an equivalent definition of $J_{T}(x)$ can be
topologically stated by the way that $J_{T}(x)$ consists of those
$y\in X$ such that for every pair of neighborhoods $ U, V $ of $ x,
y$ respectively and for each positive integer $N$, there exists a
positive integer $ n> N$ such that $$T^n U \cap V\neq \emptyset.$$
Note that on a separable Banach space $X$, an operator $T$ is
topologically transitive if and only if $J_{T}(x)=X$ for some cyclic
vector $x\in X$, see \cite{man} for more details. In \cite{cos}, G.
Costakis and A. Manoussos have characterized the $J$-class
($J^{mix}$- class) unilateral weighted shifts on
$\ell^{\infty}(\mathbb{N})$ in terms of their weight sequences
through the use of the extended limit sets. Also they proved that a
bilateral weighted shift on $\ell^{\infty}(\mathbb{Z})$ cannot be a
$J$-class. The existence of the $J$-class operators on Banach spaces
has been studied in \cite{nas2}. Some other works like as \cite{as,
az3, azmu, mar} are devoted to the study on $J$-class operators as
well.

Throughout this paper, the interior and the closure of any subset
$A$ will be denoted by $A^\circ$ and $\overline{A}$, respectively.
By $\sigma(f)$ and $ess\,sup f$ we mean the support of any function
$f\in L^p(G)$ i.e., $\{x\in G: \, f(x)\neq 0\}$ and its essential
supremum norm, respectively. The set of all continuous functions on
$G$ with compact supports will be denoted by $C_c(G)$ as well.

In this paper, by imposing special conditions on the weight
function, we aim to examine the $J$-class weighted translations $
T_{a,\omega}$ on $L^p(G)$. Moreover, we will provide some examples
on which $T_{a,\omega}$ is $J$-class but it fails to be hypercyclic.

\section{\textbf{$J$-class  weighted translation operators on  $ L^p(G) $}}
 In this section,  we shall provide the necessary and sufficient conditions for
 weighted translation operators to  possess nonzero $J$-vectors
  in the case of non torsion elements. Furthermore,
 we show that for the torsion elements of  locally compact
 groups, weighted translation operators unlike hypercyclicity
 (\cite[Lemma 1.1]{chu}), the extended limit set  $J_{T_{a,\omega}}(0)$
  can fill whole the space $L^p(G)$.


\begin{thm}\label{t2.2}
Let $G$ be a locally compact group, $a\in G$ be a strongly runaway
element  and $\omega:G\to (0,\infty)$ be a weight function. Assume
that an operator $T_{a,\omega}$ is $J$-class.  Then for each compact
subset $\Delta \subset G$ with the positive measure, there is a
sequence of subsets $(E_k)$ in $\Delta$ and a strictly increasing
subsequence $(n_k)$ of positive  integers such that
$\lambda(E_k)\to\lambda(\Delta)$ and
$$
\lim_{k\rightarrow \infty}{ess\,sup}\,
 \tilde{\omega}_{n_k}|_{E_k}=0.
$$
\end{thm}
\begin{proof}
Suppose that $T_{a,\omega}$ is a  $J$-class operator. Then there
exists a nonzero $f \in L^p(G)$ such that
$J_{T_{a,\omega}}(f)=L^p(G)$. Let $0<\epsilon<1$ and let $\Delta$ be
a compact subset with positive measure. Choose $0<\eta<1$ with
$4\eta^p<\lambda(\Delta)$ and $\frac{2\eta}{1-\eta}<\epsilon$ so
that the set $$A=\{x\in\sigma(f):|f(x)|>\eta\}$$ is of positive
finite measure. By regularity of the Haar measure $\lambda$ we may
find a compact subset $B\subseteq A$ such that $\lambda(A\setminus
B)<\eta^p$. Since an element $a$ is strongly runaway, there exists
an $N\in\mathbb{N}$ such that
\begin{equation}\label{e1}
    (\Delta\cup B)a^{-m}\cap(\Delta\cup B)=\emptyset
\end{equation}
holds eventually for each $m\geq N$. That
$J_{T_{a,\omega}}(f)=L^p(G)$, there exist a function $g \in L^p(G)$
and an $n$, of course $n>N$ such that
$$\|f-g\|_p<\delta^2, ~ \mbox{and} ~~
\|T_{a,\omega}^n(g)-\chi_{\Delta}\|_p<\delta^2,$$ where
$0<\delta<\eta$ is chosen.  Now, consider the set $$C=\{x\in
G\setminus B: |f(x)-g(x)|\geq\delta\}.$$ If $C$ is not the null set,
then
 $\lambda(C)<\delta^p$, because of
$$
\delta^{2p}>\|f-g\|_{p}^p\geq\int_C|f(x)-g(x)|^p
d\lambda(x)\geq\delta^p\lambda(C).
$$
Also, set
$$D=\{x\in\Delta:|\tilde{\omega}_n(x)^{-1}g(xa^{-n})-1|\geq\delta\}.$$
In the case $D\neq \emptyset$, we have
$$
\delta^{2p}>\|T_{a,\omega}^n(g)-\chi_{\Delta}\|_{p}^p\geq\int_D|
\tilde{\omega}_n(x)^{-1}g(xa^{-n})-1|^p
d\lambda(x)\geq\delta^p\lambda(D).
$$
Now, for each $$x\in E:=\Delta\setminus \big(D\cup
Ca^n\cup(A\setminus B)a^n\big)$$ we have $\lambda(\Delta\setminus
E)<3\eta^p$ and furthermore
$$
\tilde{\omega}_n(x)<\frac{|g(xa^{-n})|}{1-\delta}
<\frac{|f(xa^{-n})|+\delta}{1-\delta}<\frac{2\eta}{1-\eta}<\epsilon.
$$
Note that the first inequality of the above statement comes from the
fact $x\notin D$, the second from $xa^{-n} \notin C$, and the last
one, comes from $x\in \Delta \subseteq \Delta \cup B$ and then
neither $x\notin  (\Delta\cup B)a^n$ (see, (\ref{e1})) nor $x\notin
Ba^n$, since $Ba^n\subseteq (\Delta\cup B)a^n$. Together with
$xa^{-n}\notin A\setminus B$, we get that $xa^{-n}\notin A$. Finally
by choosing $\delta=\frac{1}{k}$, of course for sufficiently large
$k\in \mathbb{N}$, the desired subsequence $(n_k)$ is obtained and
the proof is complete.
\end{proof}
\begin{rem}
It is worth noting that on the support of the nonzero $J$-vector
$f$, set $f=\lim_{m\to\infty}f\chi_{A_m}$ where $A_m=\{x\in
\sigma(f): |f(x)|\geq\frac{1}{m}\}$. Pick $g\in L^p(G)$ and
$n\in\mathbb{N}$ on which $\|f-g\|_p<\delta^2$ and
$\|T_{a,\omega}^{n}g\|<\delta^2$. Set $B_m=\{x\in A_m:
|f(x)-g(x)|\geq\delta\}$. Then similarly
$0\leq\lambda(B_m)<\delta^p$. Also, by implementing the change of
variable formula to the subset   $C_m=\{x\in A_m:
|\omega_n(x)g(x)|\geq\delta\}$ we get that
$0\leq\lambda(C_m)<\delta^p$. Finally, for every $x\in E_m:=
A_m\setminus(B_m\cup C_m)$ these are inferred that
$\lambda(A_m\setminus E_m)<2\delta^p$ and
$$
\omega_n(x)<\frac{\delta}{|g(x)|}
<\frac{\delta}{\frac{1}{m}-\delta}<\epsilon.
$$
\end{rem}
In the following theorem, we aim to explore a condition that is
sufficient  for the weighted translations to be $J$-class on
$L^p(G)$-spaces. Moreover, it precisely delineates the boundary
between $J$-class and hypercyclic behavior for such operators.
\begin{thm}\label{t4}
Let $G$ be a locally compact group, $a\in G$ be a strongly runaway
element and $\omega:G\to (0,\infty)$ be a weight function. Assume
that
\begin{itemize}
  \item [(i)]for each compact subset $\Delta \subset G$ with
 positive measure,  there exist a sequence of Borel subsets $(E_k)$ in
 $\Delta$ and a strictly increasing subsequence $(n_k)$ of positive
 integers
such that $\lambda(E_k)\to\lambda(\Delta)$ and  $\lim_{k\rightarrow
\infty}{ess\,sup}\,
 \tilde{\omega}_{n_k}|_{E_k}=0$;
  \item [(ii)] for some compact subset $K \subset G$ with
$\lambda(K)>0$,
 $$ \lim_{n\rightarrow \infty}{ess\,sup}\,
 {\omega}_{n}|_K=0.$$
\end{itemize}
Then an operator $T_{a,\omega}$ is $J$-class with the $J$-vector
$\chi_K$.
\end{thm}
\begin{proof}
Assume that the statements of $(i)$ and $(ii)$ hold. We shall prove
that $J_{T_{a,\omega}}(\chi_K)=L^p(G)$. Since
$\overline{C_c(G)}=L^p(G)$ then for each  $f\in C_c(G)$, we are
going to find a sequence
 $(g_k)$ in $L^p(G)$ and a strictly increasing sequence of positive integers
  $(n_k)$  such that $\lim_{k\rightarrow \infty}\|g_k - \chi_K
\|_p=0$ and $\lim_{k\rightarrow
\infty}\|T_{a,\omega}^{n_k}(g_k)-f\|_p=0$.
 For an arbitrary nonzero $f\in C_c(G)$ we know that
$\sigma(f)$ is a non-empty and compact subset of $G$.
 Since the element $a$ is strongly runaway,
  one may find an $m\in\mathbb{N}$ such that $$\sigma(f) \cap
\sigma(f)a^{\pm n}=\emptyset, ~ \mbox{for all}~ n>m.$$
 Moreover, for each $\epsilon>0$, according to our hypotheses
 there exist a  Borel subset $E_k \subseteq \sigma(f)$ and a positive integer
  $n_k>m$, for an appropriate $k$ depending on $\epsilon$  such that
 $$\lambda(\sigma(f)-E_k)<\frac{\epsilon} {2^p \Vert
f\Vert_{\infty}^{p}}, ~~
ess\,sup\,\tilde{\omega}_{n_k}<\frac{\epsilon}{\|f\|_{p}^p} ~
\mbox{on} ~ E_k$$ and
$$ess\,sup\,\omega_{n_k}<\frac{\epsilon}{2^p\|f\|_{p}^p} ~ \mbox{on}~
K.$$ Now, define the Borel measurable function $g_k$ by
$$g_k(x):=\tilde{\omega}_{n_k}(xa^{n_k})f\chi_{E_k}(xa^{n_k})+\chi_{K}(x), ~~~\, x\in G.$$
The Borel subset $E_k$ is relatively compact and this fact ensures
that $g_k\in L^p(G)$.  In the sequel, we are going to show that
$\|g_k -\chi_K\|_p<\epsilon$ and
$\|T_{a,\omega}^{n_k}(g_k)-f\|_p<\epsilon$. To this end, by the
change of variable formula and the aforementioned facts we have
\begin{eqnarray*}
 \|g_k-\chi_{K}\|_p^p&=&\int_G|\tilde{\omega}_{n_k}(xa^{n_k})
 f\chi_{E_k}(xa^{n_k})|^p d\lambda(x) \\
  &=&\int_G|\tilde{\omega}_{n_k}(x)f\chi_{E_k}(x)|^p d\lambda(x) \\
&=&\int_{E_k}|\tilde{\omega}_{n_k}(x)f(x)|^p d\lambda(x)\\
      &<& \epsilon.
  \end{eqnarray*}
On the other hand, we have
\begin{eqnarray*}
&&\Vert T_{a,\omega}^{n_k}(g_k) - f\Vert_p^p
\\&=&
\int_G \vert T_{a,\omega}^{n_k}(g_k)(x) -f(x)\vert^p d\lambda(x)
\\&=&
\int_G \vert
\omega(x)\omega(xa^{-1})\cdot\cdot\cdot\omega(xa^{-n_k+1})g_k(xa^{-n_k})
-f(x)\vert^p d\lambda(x)
\\&=&
\int_G \vert
\omega(x)\omega(xa^{-1})\cdot\cdot\cdot\omega(xa^{-n_k+1})
\tilde{\omega}_{n_k}(x)f(x)\chi_{E_k}(x)\\&& +
\omega(x)\omega(xa^{-1})\cdot\cdot\cdot\omega(a^{-n_k+1}x)\chi_K(xa^{-n_k})
-f(x)\vert^p d\lambda(x)
\\&\leq&
2^p\int_G|f\chi_{E_k}(x)-f(x)|^p d\lambda(x)
\\&&+
2^p\int_G
\vert\omega(x)\omega(a^{-1}x)\cdot\cdot\cdot\omega(xa^{-n_k+1})\chi_K(xa^{-n_k})
\vert^p d\lambda(x)
\\&\leq&
2^p\lambda\big(\sigma(f)-E_k\big)\|f\|_{\infty}^p + 2^p\int_G
\vert{\omega}_{n_k}(x)\chi_K(x) \vert^p d\lambda(x)\\&<& \epsilon.
\end{eqnarray*}
Finally, by proceeding the above arguments inductively, we can
construct a  strictly increasing subsequence $(n_k)$ of positive
integers and a sequence $(g_k)$ associated to a sequence
$(\epsilon_k)$ of positive reals convergent to $0$, on which the
desired result is established.
\end{proof}
\begin{thm}
Let $\omega$ be a  continuous weight on a locally compact group $G$
and $\frac{1}{\omega}\in L^{\infty}(G)$. Then for each torsion
element $a\in G,$ the following conditions are equivalent:
\begin{itemize}
\item[(i)] $J_{T_{a,\omega}}(0)=L^p(G)$;\\
\item[(ii)] for every compact set $F\subseteq G$ of positive measure and
every $\epsilon, \delta>0$ there exist $n\in\mathbb{N}$ and a Borel
subset $E\subseteq F$ such that $\lambda(F\setminus E)<\delta$ and
$$
{ess\,sup}\, \omega_n|_{E}^{-1}<\epsilon.
$$
\end{itemize}
\end{thm}
\begin{proof}
Fix an arbitrary torsion element $a \in G$ of the order $\gamma$. To
prove the implication $\mbox{(i)}\Rightarrow\mbox{(ii)}$, suppose
that $J_{T_{a,\omega}}(0)=L^p(G)$. Take $\epsilon, \delta>0$ and a
compact subset $F\subseteq G$ arbitrarily. Without loss of
generality, assume that $\lambda(F)>3\delta^{p}$ and
$\frac{\delta}{1-\delta}<\epsilon$. Hence, we may find $g\in L^p(G)$
and $n\in\mathbb{N}$ large enough satisfying $\|g\|_p<\delta^2$ and
$$
\|T^n_{a,\omega}g-\chi_{\bigcup_{k=0}^{\gamma-1}Fa^k}\|_p<\delta^2.
$$
Set $A=\{x\in F: |g(x)|\geq\delta\}$. In the case $A$ is a non-empty
subset, it is inferred that
$$
\delta^{2p}>\|g\|_{p}^p\geq\int_A |g(x)|^p
d\lambda(x)\geq\delta^p\lambda(A),
$$
that is $\lambda(A)<\delta^p$. Also, set
$$
B=\left\{x\in F:\,\left|\omega_n(x)g(x)-1\right|\geq\delta\right\}.
$$
Now, we may write $n=\gamma[\frac{n}{\gamma}]+k$ where $0\leq
k\leq\gamma-1$. If $\lambda(B)\neq 0$, then by the  change of
variable formula  we obtain that
\begin{eqnarray*}
\delta^{2p} &>&
\|T^{n}_{a,\omega}g-\chi_{\bigcup_{k=0}^{\gamma-1}Fa^k}\|_{p}^p\\
&=& \int_G |\omega(x)\omega(xa^{-1})\cdots
\omega(xa^{-(n-1)})g(xa^{-n})-\chi_{\bigcup_{k=0}^{\gamma-1}Fa^k}(x)|^p d\lambda(x)\\
&=& \int_G |\omega(xa^{n})\omega(xa^{n-1})\cdots
\omega(x a)g(x)-\chi_{\bigcup_{k=0}^{\gamma-1}Fa^k}(xa^n)|^p d\lambda(x)\\
&\geq&
\int_B |\omega_n(x)g(x)-1|^p d\lambda(x)\\
&\geq& \delta^p\lambda(B).
\end{eqnarray*}
Therefore, $\lambda(B)\leq\delta^p$ and subsequently
$\lambda\big(F\setminus(A\cup B)\big)>\delta^p$. Finally, by setting
$E:=F\setminus(A\cup B)$ we obtain that $\lambda(F\setminus
E)<2\delta^p$ and in turn for each $x\in E$ we have
$$
\frac{1}{\omega_n(x)}<\frac{|g(x)|}
{1-\delta}<\frac{\delta}{1-\delta}<\varepsilon.
$$

To  prove the converse implication
$\mbox{(ii)}\Rightarrow\mbox{(i)}$, the initial definition of
$J_{T_{a,\omega}}(0)$ does not use here any longer. Define the
operator $S_{a,\omega}: C_c(G)\rightarrow L^p(G)$ by
$$ S_{a,\omega}(h)(x)=\frac{h(xa)}{\omega(xa)}, ~~ \mbox{for each}~ x\in G~ \mbox{
and}~~ h\in C_c(G).$$ There is no doubt that the operator
$S_{a,\omega}$ is well-defined since the weight $\omega$ is assumed
to be bounded away from zero and $\frac{1}{\omega}\in
L^{\infty}(G)$.  Take open neighborhood $U$ of the origin and a
non-empty open neighborhood $V$ in $L^p(G)$ arbitrarily. As the set
$C_c(G)$ is dense in $L^p(G)$, there is a $g\in C_c(G)\cap V$ and
hereon let $F=\sigma(g)$. Since $\omega$ is continuous, then for
some $\eta>0$,
$$\inf\{\prod_{i=1}^{k}\omega(xa^{i})^{-1}: 1\leq
k\leq\gamma-1\}>\eta, ~ ~ \mbox{for all} ~ x\in F.$$ Now by the
assumptions in $\mbox{(ii)}$, for each $\epsilon,\delta>0$, there
exist $E\subseteq F$ and a positive integer $n$ satisfying
$$\lambda(F\setminus E)<\delta ~ \mbox{and}~ \underset{x\in
E}{ess\,sup}\,\big(\prod_{i=1}^{n}\omega(xa^{i})\big)^{-1}<\epsilon.$$
In particular, if we consider
$\omega_{\gamma}(x)=\omega(x)\omega(xa)\cdots\omega(xa^{\gamma-1})$,
then apparently $\underset{E}{ess\,sup}\,\omega_{\gamma}>1$. We may
also write  $n=\gamma[\frac{n}{\gamma}]+k$ where $0\leq
k\leq\gamma-1$, which implies  that
\begin{eqnarray*}
\|S_{a,\omega}^{\gamma n}(g\chi_{E})\|_{p}^{p} &=&
\int_G\left|\frac{g(xa^{\gamma n})\chi_{E}(xa^{\gamma n})}{\omega(x a)\cdots\omega(xa^{\gamma n})}\right|^p d\lambda(x)\\
&=&
\int_{E}\left|\frac{g(x)}{\omega(x a)\cdots\omega(xa^{\gamma n})}\right|^p d\lambda(x)\\
&\leq&
(\underset{}{ess\,sup}\,\omega_{\gamma}^{-n})^p\,(\underset{}{ess\,sup}\,
g)^p\,\lambda(E).
\end{eqnarray*}
Moreover,
$$
\|T_{a,\omega}^{\gamma n}\big(S_{a,\omega}^{\gamma
n}(g\chi_{E})\big)-g\|_p= \|g\chi_{E}-g\|_p\leq
ess\,sup\,g\lambda(F\setminus E)<\delta \,{ess\,sup}\,g.
$$
The last ensures  that $T_{a,\omega}^{\gamma n}(U)\cap
V\neq\emptyset$. On the other hand, in the case
${ess\,sup}\,\omega_{\gamma} \leq 1$, it should be noted that
\begin{eqnarray*}
  \|T^{n}_{a,\omega}g\|_{p}^{p}&\leq&
ess\,sup\,\omega_{\gamma}^{p[\frac{n}{\gamma}]}\int_G
\omega^p(ax)\cdots\omega^p(a^{k}x)|g(x)|^p d\lambda(x) \\
&\leq&
ess\,sup\,\omega_{\gamma}^{p[\frac{n}{\gamma}]}ess\,sup\,\omega^{pk}\|g\|_p^p.
\end{eqnarray*}
Hence, the sequence $\{T^{n}_{a,\omega}g\}_{n=1}^{\infty}$ is
bounded for every $n\in\mathbb{N}$ and each $g\in L^p(G)$. Actually,
in this case, $T_{a,\omega}$ is power bounded and then by
\cite[Proposition 2.10]{man},
$J_{T_{a,\omega}}(0)=L_{T_{a,\omega}}(0)\neq L^p(G)$.
\end{proof}
\begin{cor}
Let $a$ be a torsion element of a locally compact group $G$. Then
$0\in J_{T_{a,\omega}}(0)^\circ$ if and only if
$J_{T_{a,\omega}}(0)=L^p(G)$.
\end{cor}
\begin{proof}
Suppose that  $0\in J_{T_{a,\omega}}(0)^\circ$. Then, there is a
positive real number $r$ such that $B_{r}(0)\in
J_{T_{a,\omega}}(0)$. Hence we can choose $\delta>0$ sufficiently
small with $2\delta\chi_{\cup_{k=0}^{\gamma-1}Fa^k}\in B_{r}(0)$. It
follows that there exist $g\in L^p(G)$ and $n\in\mathbb{N}$ such
that $\|g\|_p<\delta^3$ and
$\|T^n_{a,\omega}g-2\delta\chi_{\cup_{k=0}^{\gamma-1}Fa^k}\|_p<\delta^2$.
Therefore, similar to the proof of the preceding theorem the rest is
done.
\end{proof}
\begin{thm}\label{t3}
Let $G$ be a  locally compact group. Assume that  an element $a\in
G$ is strongly runaway.  Then the followings are equivalent:
\begin{itemize}
  \item [(i)] $J_{T_{a,\omega}}(0)=L^p(G)$,
  \item [(ii)]for each compact subset $\Delta \subset G$ with
 positive measure,  there exist a sequence of Borel subsets $(E_k)$ in
 $\Delta$ and a strictly increasing subsequence $(n_k)$ of positive
 integers
such that $\lambda(E_k)\to\lambda(\Delta)$ and  $$\lim_{k\rightarrow
\infty}{ess\,sup}\,
 \tilde{\omega}_{n_k}|_{E_k}=0.$$
\end{itemize}
\end{thm}
\begin{proof}
${(i)}\Rightarrow {(ii)}$. Suppose that
$J_{T_{a,\omega}}(0)=L^p(G)$. Fix an arbitrary compact subset
$\Delta\subset G$ with $\lambda(\Delta)>0$. Since $a$ is a strongly
runaway element,  there exists an $N\in\mathbb{N}$ such that
$$\Delta \cap \Delta a^{\pm n}=\emptyset, ~~  \mbox{for all} ~~
n>N.$$ By definition of $J_{T_{a,\omega}}(0)$, for each $\epsilon>0$
we may find a function $g_k\in L^p(G)$ and a positive integer $n_k$,
of course for an appropriate $k$ depending on $\epsilon$ such that
$$\|g_k \|_p <\eta^2 ~ \text{ and}~
\|T_{a,\omega}^{n_k}(g_k)-\chi_\Delta\|_p<\eta^2,$$ in which $ \eta>
0$, is chosen in the way that $\frac{\eta}{1-\eta}<\epsilon$ and
$\lambda(\Delta)>3\eta^p$. Further, without loss of generality, we
may assume that $n_k>N$. Now define the set  $$B_{\eta,k}=\{x\in G:
|g_k(x)|\geq\eta\}.$$ Hence, $0\leq\lambda(B_{\eta,k})<\eta^p$. On
the other hand, assume that
$$C_{\eta,k}=\{x\in\Delta:|\tilde{\omega}_{n_k}(x)^{-1}g_k(xa^{-n_k})-1|\geq\eta\}.$$
In the case $C_{\eta,k}\neq \emptyset$,  we obtain that
$0<\lambda(C_{\eta,k})<\eta^p$. Eventually, by setting
$$E_k:=\Delta\setminus \big(C_{\eta,k}\cup B_{\eta,k}a^{n_k}\big),$$
we are approaching to the desired result. Indeed, we have
$\lambda(\Delta\setminus E_k)<\eta^p$ and
$$
\tilde{\omega}_{n_k}(x)<\frac{|g_k(xa^{-n_k})|}{1-\eta}
<\frac{\eta}{1-\eta}<\epsilon.
$$
Finally, by proceeding the above arguments inductively, we can
construct a  strictly increasing subsequence $(n_k)$ of positive
integers associated to a sequence $(\epsilon_k)$ of positive reals
and convergent to $0$, on which the result is established.

${(ii)}\Rightarrow {(i)}$. We shall prove that
$J_{T_{a,\omega}}(0)=L^p(G)$. Indeed, for each  $f\in L^p(G)$, we
are going to find a sequence
 $(g_k)$ in $L^p(G)$ and a strictly increasing sequence of positive integers
  $(n_k)$  such that $\lim_{k\rightarrow \infty}\|g_k
\|_p=0$ and $\lim_{k\rightarrow
\infty}\|T_{a,\omega}^{n_k}(g_k)-f\|_p=0$.
 Note that
$\overline{C_c(G)}=L^p(G)$ and so it is enough to prove the last
assertion for each function of $C_c(G)$. Fix an arbitrary non-zero
function $f\in C_c(G)$. Clearly, $\sigma(f)$ is a non-empty and
compact subset of $G$. Hence, for each $\epsilon>0$, according to
the assumptions,  there exist a Borel subset $E_k\subseteq
\sigma(f)$ and a positive integer $n_k$, for an appropriate $k$ but
depending on $\epsilon$ such that
$$\lambda(\sigma(f)-E_k)<\frac{\epsilon} { \Vert
f\Vert_{\infty}^{p}}~ \mbox{and} ~
 \underset{x\in E_{k}}{ess\,sup}\,
\tilde{\omega}_{n_k}(x)<\frac{\epsilon}{\Vert f\Vert_p^p}.$$ Define
the Borel measurable function $g_k$ by
$$g_k(x):=\tilde{\omega}_{n_k}(xa^{n_k})f\chi_{E_k}(xa^{n_k}), ~~ \, x\in G.$$
Clearly $g_k\in L^p(G)$, since the Borel subset $E_k$ is relatively
compact.  Afterwards, by the change of variable formula and the
aforementioned facts we obtain that
\begin{eqnarray*}
 \|g_k\|_p^p&=&\int_G|\tilde{\omega}_{n_k}(xa^{n_k})
 f\chi_{E_k}(xa^{n_k})|^p d\lambda(x) \\
  &=&\int_G|\tilde{\omega}_{n_k}(x)f(x)\chi_{E_k}(x)|^p d\lambda(x) \\
&=& \int_{E_k}|\tilde{\omega}_{n_k}(x)f(x)|^p d\lambda(x)\\
      &<& \epsilon.
  \end{eqnarray*}
 On the other hand, we have
\begin{eqnarray*}
&&\Vert T_{a,\omega}^{n_k}(g_k) - f\Vert_p^p
\\&=&
\int_G \vert T_{a,\omega}^{n_k}(g_k)(x) -f(x)\vert^p d\lambda(x)
\\&=&
\int_{G} \vert
\omega(x)\omega(xa^{-1})\cdot\cdot\cdot\omega(xa^{-n_k+1})\tilde{\omega}_{n_k}(x)
f\chi_{E_k}(x)-f(x)\vert^p d\lambda(x)\\
&=&\int_{G} \vert f\chi_{E_k}(x)-f(x)\vert^p d\lambda(x)\\ &\leq&
\lambda\big(\sigma(f)-E_k\big) \|f\|_{\infty}^p\\&<&\epsilon.
\end{eqnarray*}
At the end, as it is mentioned in the proof of the previous
implication, we can inductively get a  sequence
 $(g_k)$ in $L^p(G)$ and a strictly increasing sequence of positive integers
  $(n_k)$ associated to a sequence $(\epsilon_k)$ of positive reals
  convergent to $0$ such that the desired result is hold.
\end{proof}
\begin{rem}\label{r1}
Note that in the implication $(ii)\Rightarrow (i)$ of Theorem
\ref{t3}, apparently it has not been assumed that a group  $G$ is to
be second countable. As a result, $J_{T_{a,\omega}}(0)$ can fill the
non-separable $L^p(G)$ spaces while it is impossible for the
hypercyclic phenomena.
\end{rem}
\begin{thm}
Let $G$ be a locally compact group with a fixed right Haar measure
$\lambda$. Assume that $a\in G$ is strongly runaway. Then the
followings are equivalent:
\begin{itemize}
\item[(i)] $0\in J_{T_{a,\omega}}(0)^\circ$;\\
\item[(ii)] $J_{T_{a,\omega}}(0)=L^p(G)$.
\end{itemize}
\end{thm}
\begin{proof}
It is adequate to prove $\mbox{(i)}\Rightarrow\mbox{(ii)}$, as the
converse is obvious. Take $r>0$ such that the open ball $B_{r}(0)$
is contained in $J_{T_{a,\omega}}(0)$. Pick an arbitrary compact
subset $K\subset G$ with $\lambda(K)>0$ and choose $\epsilon>0$ with
$\epsilon<r$ and $\lambda(K)>3\epsilon^p$. Hence
$\displaystyle\frac{\epsilon}{2}\lambda(K)^{-\frac{1}{p}}\chi_{K}\in
B_{\epsilon}(0)$. It follows that there exist $g\in L^p(G)$ and
$n\in\mathbb{N}$ large enough, such that $K\cap Ka^{-n}=\emptyset$,
$$
\|g\|_p<\frac{\epsilon^3}{4\lambda(K)^\frac{1}{p}} \qquad \mbox{and}
\qquad
\|T_{a,\omega}^n(g)-\frac{\epsilon}{2\lambda(K)^\frac{1}{p}}\chi_{K}\|_p<
\frac{\epsilon^2}{4\lambda(K)^\frac{1}{p}}.
$$
By defining $C=\{x\in G\setminus
K:\,|g(x)|>\frac{\epsilon^2}{4\lambda(K)^\frac{1}{p}}\}$ we get that
$$
\frac{\epsilon^{3p}}{4^p\lambda(K)}>\int_G|g(x)|^p\,d\lambda(x)\geq
\frac{\epsilon^{2p}}{4^p\lambda(K)}\lambda(C).
$$
Therefore, in the case $\lambda(C)\neq 0$, it is inferred that
$0<\lambda(C)<\epsilon^p$. Also, put
$$
D=\left\{x\in K:\left|\prod_{i=0}^{n-1}
\omega(xa^{-i})g(xa^{-n})-\frac{\epsilon}{2\lambda(K)^{\frac{1}{p}}}\right|>
\frac{\epsilon}{4\lambda(K)^{\frac{1}{p}}}\right\}.
$$
If $\lambda(D)\neq 0$, then
$$
\frac{\epsilon^{2p}}{4^p\lambda(K)}>
\int_G|T_{a,\omega}^{n}g(x)-\frac{\epsilon}{2\lambda(K)^\frac{1}{p}}\chi_{K}(x)|^p
d\lambda(x)\geq \frac{\epsilon^{p}}{4^p\lambda(K)}\lambda(D),
$$
whence $0<\lambda(D)<\epsilon^p$. Now for each $x\in
K\setminus(D\cup Ca^{n})$ we obtain that
$$
\frac{1}{\prod_{i=0}^{n-1}\omega(xa^{-i})}<\frac{4\lambda(K)^{\frac{1}{p}}}
{\epsilon}|g(xa^{-n})|<\epsilon.
$$
\end{proof}

\begin{exam}
Let $G=(\mathbb{R}, +)$, be the group of  real numbers equipped with
the  Lebesgue measure.  Define the weight function $w$ on
$\mathbb{R}$ by
$$w(x)=\left\{
  \begin{array}{ll}
   \alpha, & {1\leq x}, \\
        -\frac{x}{2}+1, & {-1< x < 1},\\
    \beta, &{x\leq -1} ,
  \end{array}
\right.$$  where the real numbers $\alpha,\beta$ are arbitrarily
chosen such that $1<\alpha <\beta$. Fix a  nonzero number $a\in
\mathbb{R}$ and consider the convolution
$(\omega*\delta_{a})(x)=\omega(x- a)$, for each $x\in \mathbb{R}$.
The statement of Theorem \ref{t3} ensures that
$J_{T_{a,\omega}}(0)=L^p(\mathbb{R})$ while $T_{a,\omega}$ is not
hypercyclic operator according \cite[Theorem 2.2]{chu}. Every
compact subset $\Delta \subseteq \mathbb{R}$   may embed in the
interval $[-\ell, \ell]$ for instance, because of its boundedness.
Now, after finitely many iterations (say, $m$), it will be shifted
backward (towards the left side of the real line) and lies in
$(-\infty, -1]$ eventually under the action of the weighted
translation $T_{a, \omega}$ with $a>0$. Then, for each $x\in \Delta$
and with a suitable strictly increasing subsequence $(n_k)$ we have
that
\begin{eqnarray*}
 {\tilde{\omega}_{n_k}}(x)&=&\frac{1}{\prod_{i=0}^{n_k-1}\omega(x-ia)} \\
   &=& \frac{1}{\prod_{i=0}^{m}\omega(x-ia)}\frac{1}{\prod_{i=m+1}^{n_k-1}\omega(x-ia)} \\
   &\leq& C \frac{1}{\beta^{n_k-m-1}}\rightarrow 0, ~~ \mbox{as} ~ k\rightarrow \infty. \\
  \end{eqnarray*}
 The existence of the constant $C>0$ is due to the boundedness of
$\omega$.

 While by repeating a similar
argument, but this time by shifting forward to reach $[1, +\infty)$
eventually and then we get that
\begin{eqnarray*}
   \omega_{n}(x)&=&\prod_{i=1}^{n}\omega(x+ia)\\
   &=& \prod_{i=1}^{m'}\omega(x+ia)\prod_{i=m'+1}^{n}\omega(x+ia) \\
  &\leq& [\prod_{i=1}^{m'}\omega(x+ia)]{\alpha^{n-m'+1}}
   \rightarrow \infty (\nrightarrow 0), ~~ \mbox{as} ~ n\rightarrow \infty.\\
  \end{eqnarray*}

 \end{exam}
\begin{exam}
Let $G=(\mathbb{R}, +)$, be the group of the real numbers equipped
with  Lebesgue measure.  Define the weight function $w$ on
$\mathbb{R}$ by
$$\omega(x)=\left\{
  \begin{array}{ll}
   2, & {x\leq -1}, \\
        -\frac{7}{4}x+\frac{1}{4}, & {-1< x \leq 0},\\
    \frac{7}{4}x+\frac{1}{4}, &{0<x\leq 1} ,\\
-\frac{7}{4}(x-1-m)+\frac{1}{4}, & {m< x \leq m+1 (m=1,3,5,...)},\\
\frac{7}{4}(x-m)+\frac{1}{4}, & {m< x \leq m+1 (m=2,4,6,...)}.
  \end{array}
\right.$$ Indeed, for each $x>2$ we adopt that
$\omega(x)=\omega(x-2)$. See the following figure:

\begin{figure}[h]
  \centering
 \includegraphics[width=0.5\textwidth]{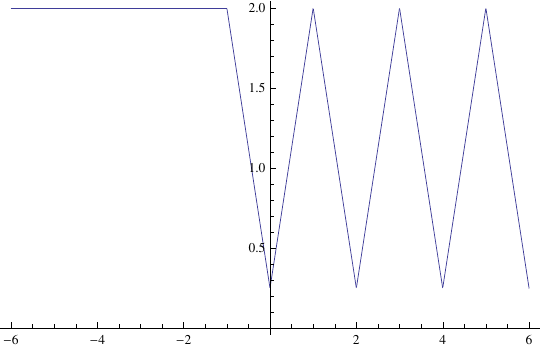}
  \caption{Weight function}
  \label{fig:10}
\end{figure}

 For a  nonzero number $a\in
\mathbb{R}$  consider the convolution
$(\omega*\delta_{a})(x)=\omega(x- a)$, for each $x\in \mathbb{R}$
and $f\in L^p(\mathbb{R})$.   Fix $a=2$ and choose  compact subset
$\Delta\subseteq \mathbb{R}$ arbitrarily. Since $\Delta$ is bounded,
it may embed in the interval $[-\ell, \ell]$ for instance. After
finitely many iterations (say, $m$), it will be shifted backward
(towards the left side of the real line) and lies  in $(-\infty,
-1]$ eventually under the action of the weighted translation $T_{2,
\omega}$. Then, for each $x\in \Delta$ and with a strictly
increasing subsequence $(n_k)$ we have that
\begin{eqnarray*}
  {\tilde{\omega}_{n_k}}(x)&=&\frac{1}{\prod_{i=0}^{n_k-1}\omega(xa^{-i})} \\
   &=& \frac{1}{\prod_{i=0}^{n_k-1}\omega(x-2{i})} \\
   &=& \frac{1}{\prod_{i=0}^{m}\omega(x-2{i})}
\frac{1}{\prod_{i=m+1}^{n_k-1}\omega(x-2{i})} \\
   &\leq& C \frac{1}{2^{n_k-m-1}} \rightarrow 0, ~~ \mbox{as} ~ k\rightarrow \infty.\\
\end{eqnarray*}
The existence of an aforementioned constant $C>0$ follows from the
boundedness of $\omega$.  Also, for each $x\in [0, \frac{1}{4}]$ we
have $\frac{1}{4}\leq \omega(x)\leq \frac{11}{16}$. Hence,
$$\omega_{n}(x)=\prod_{i=1}^{n}\omega(xa^{i})=
\prod_{i=1}^{n}\omega(x+2{i}) \rightarrow 0,$$ as $n\rightarrow
\infty$. For this, note that by shifting forward the compact
interval $[0, \frac{1}{4}]$ with  $T_{2, \omega}$, the estimate
$\frac{1}{4}\leq \omega(x) \leq \frac{11}{16}$ will be saved in the
end.

 Eventually, the conditions of Theorem \ref{t4} are satisfied  and
 then it  guarantees  that the corresponding weighted translation $T_{2,w}$ is $J$-class
on $L^p(\mathbb{R})$ while it is not hypercyclic by \cite[Theorem
2.2]{chu}. To this end, consider that on the compact subset
$[\frac{1}{2}, 1]$ we have $\omega(x)\geq \frac{9}{8}$ and hence
$$\omega_{n}(x)=\prod_{i=1}^{n}\omega(x+2{i}) \rightarrow \infty (\nrightarrow 0).$$
 \end{exam}
\section{conclusions}
There are $J$-class weighted translations on locally compact groups
which are not hypercyclic. Indeed, according to Theorem \ref{t4}, we
may find many nonzero functions of $L^p(G)$ which are $J$-vectors
for non hypercyclic weighted translations on $L^p(G)$(\cite[Theorem
2.2]{chu}). Also,  $J$-class weighted translations  can occur on
non-separable $L^p(G)$ spaces whereas hypercyclic phenomena cannot
(see, Remark \ref{r1}). As we have seen in the results, the extended
limit sets i.e., $J_T(x)$ are playing powerful role rather than the
orbits in the dynamics of linear operators. For this, note that  the
orbit of weighted translation $T_{a,\omega}$ at zero cannot be whole
space $L^p(G)$ while it is possible $J_{T_{a,\omega}}(0)=L^p(G)$.
Another difference between $J$-sets and orbits of weighted
translations on locally compact groups, is in the case that  a right
translated element is torsion. In this case, unlike the case of
non-dense orbits of weighted translation operators (\cite[Lemma
1.1]{chu}), we have $J_{T_{a,\omega}}(0)=L^p(G)$.

\section*{Declarations}

\textbf{Competing Interests.} The authors declare that they have no
competing interests.

\textbf{Funding Information.} This research received no external
funding.

\textbf{Author Contributions.} All authors contributed equally to
the conception, development, and writing of this manuscript. All
authors read and approved the final version of the manuscript.

\textbf{Data Availability Statement.} Not applicable. No data sets
were generated or analyzed during the current study.



\end{document}